\documentclass[preprint,11pt]{elsarticle}
\usepackage{amssymb}
\usepackage{url}
\usepackage{hyperref}
\hypersetup{colorlinks,linkcolor={blue},citecolor={blue},urlcolor={red}}  
\usepackage{amsmath}
\usepackage{algorithm}
\usepackage{color}
\usepackage{algpseudocode}
\usepackage{amsthm,setspace}
\newtheorem{theorem}{Theorem}
 \usepackage{lscape}
 
\newtheorem{corollary}{Corollary}
\usepackage[margin=2.8cm]{geometry}
\setcitestyle{authoryear,open={(},close={)}}
\usepackage{lineno}
\makeatletter
\def\ps@pprintTitle{%
 \let\@oddhead\@empty
 \let\@evenhead\@empty
 \def\@oddfoot{\reset@font\hfil\thepage\hfil}
 \let\@evenfoot\@oddfoot
}
\makeatletter

\begin{document}
\begin{frontmatter}

\title{Empirical likelihood ratio tests for Cauchy distribution based on characterization }

\author[label1]{Ganesh Vishnu Avhad }
 \ead{avhadgv@gmail.com}
\author[label1]{Ananya Lahiri}
\author[label2]{Sudheesh K. Kattumannil}  
\address[label1]{Department of Mathematics and Statistics, Indian Institute of Technology, Tirupati, Andhra Pradesh, India}
\address[label2]{Statistical Sciences Division, Indian Statistical Institute, Chennai, Tamil Nadu, India}

 \doublespacing
\begin{abstract}
 Heavy-tailed distributions, such as the Cauchy distribution, are acknowledged for providing more accurate models for financial returns, as the normal distribution is deemed insufficient for capturing the significant fluctuations observed in real-world assets. In this paper, we develop a goodness-of-fit test for the Cauchy distribution based on the characterization. The asymptotic distribution of the test statistic is obtained under the null hypothesis.  In addition to that, we also proposed an empirical likelihood-based test, including the jackknife empirical likelihood (JEL) and the adjusted jackknife empirical likelihood (AJEL). An extensive Monte Carlo simulation studies are conducted to evaluate the finite sample performance of the proposed tests. The simulation results show that the proposed tests have good power compared to others. Finally, the application of the proposed tests is illustrated through the analysing two real data sets.
\end{abstract}

\end{frontmatter}
\doublespacing
\section{Introduction}\label{sec1}
\noindent The Cauchy distribution is widely used in statistical analysis, particularly within finance and economics, due to its relevance in modelling financial returns and economic variables that exhibit extreme values or outliers. Its heavy-tailed nature makes it a valuable distribution, but the lack of a finite mean and variance introduces specific challenges in statistical modelling and hypothesis testing. Numerous applications of the Cauchy distribution are documented in the literature; for instance, it has been utilized across various fields, such as modelling the behaviour of polar and non-polar liquids in porous glasses within the mechanical and electrical theory (\cite{stapf1996proton}), addressing physical anthropology and measurement problems (\cite{johnson1995continuous}), and modelling value-at-risk in risk and financial analysis (\cite{liu2012intermediate}). Thus, a goodness-of-fit test for assessing whether the underlying distribution conforms to a Cauchy form would be extremely useful in practical applications. Recently, based on various characterizations, several authors, including \cite{mahdizadeh2017new}, \cite{mahdizadeh2019goodness}, \cite{villasenor2021goodness}, and \cite{pekgor2023novel}, have proposed goodness-of-fit tests for the Cauchy distribution. 


In recent years, empirical likelihood methods have gained popularity in non-parametric statistics for their flexibility and minimal reliance on strong parametric assumptions. The empirical likelihood ratio test has been particularly valued for constructing confidence regions and hypothesis tests (\cite{liu2023review}). However, it has been observed that traditional empirical likelihood ratio tests may suffer from reduced accuracy and power, especially when applied to small sample sizes or when the underlying distribution deviates from the assumption of normality. To overcome these limitations, \cite{jing2009jackknife} proposed the jackknife empirical likelihood (JEL) ratio test as a more robust alternative. Since then, various tests based on JEL have been developed and explored in the literature, see \cite{feng2012jackknife}, \cite{bhati2020jackknife} and \cite{khan2024testing}. 
The jackknife method, known for reducing bias and variance in statistical estimates, has been employed to enhance the empirical likelihood framework. In addition to addressing both the empty set issue and low coverage probability, the adjusted jackknife empirical likelihood (AJEL) method was proposed by \cite{chen2008adjusted}. This improvement has led to better accuracy of the empirical likelihood ratio test, particularly when dealing with small samples and heavy-tailed distributions such as the Cauchy distribution. In this paper, we propose both a JEL and an AJEL ratio test for the Cauchy distribution.

This manuscript is structured as follows: In Section \ref{sec2}, the test statistic, empirical likelihood methods, JEL and AJEL and their asymptotic distribution are presented. The power properties of these tests are evaluated using Monte Carlo simulations, and the results are presented in Section \ref{SStudy}. Section \ref{sec4} illustrates the proposed methods by analysing real data sets. Finally, Section \ref{conclusion} concludes with a summary.

\section{ Test statistics}\label{sec2}
\noindent A random variable $X$ is said to have a standard Cauchy distribution, denoted by $X\sim \mathcal{C}(0,1)$, if it has a probability density function (p.d.f.)
\begin{equation}\label{pdf}
    f(x)= \dfrac{1}{\pi(1+x^2)}, ~~ x\in \mathbb{R},
\end{equation}
and cumulative distribution function (c.d.f.) 
\begin{equation}\label{cdf}
    F(x)= \dfrac{1}{2}+\dfrac{1}{\pi}tan^{-1}(x),  ~~ x\in \mathbb{R}.  
\end{equation}
We are interested in testing the null hypothesis as follows
\begin{align}\label{ho}
     H_0 &: \text{ $X \in \mathrm{C}(0,1)$} \quad \text{against} \quad H_1  : \text{$X \notin \mathrm{C}(0,1)$}.
\end{align}

A new class of tests is introduced, derived from the following characterization of the Cauchy distribution.

\begin{theorem}(\cite{arnold1979some})\label{thm1}
 Let $X, X_1$ and $X_2$ be independent and identically distributed (i.i.d.) absolutely continuous random variables. Let us define $Y=(X_1- X_2^{-1})/2$. We recall $X$ has a distribution function (d.f.) $F$ and let $Y$ have a d.f. $G$. Then, $X$ and $Y=(X_1- X_2^{-1})/2$ have the same distribution if, and only if, $X \sim \mathcal{C}(0,1)$. 
\end{theorem}
 This characterization is classified within the category of distributional characterizations based on equality. Constructing tests based on such characterizations has become popular in goodness-of-fit testing because they ensure equivalence across distribution functions, density functions, characteristic functions, and related analytical properties.

Consider a random sample $X_1, X_2, \ldots, X_n$ of size $n$ from an unknown distribution function ($d.f.$) $F$.
 To test the hypothesis defined in (\ref{ho}), a measure of departure, denoted as \( \Delta \), is defined as follows:
 \begin{align*}
    \Delta&=\int_{-\infty}^\infty \bigg(P\bigg(\dfrac{X_1-X_2^{-1}}{2} \leq t\bigg)-P(X_3<t)\bigg)dF(t)\\
    &=\int_{-\infty}^\infty \bigg(P\bigg(\dfrac{X_1X_2-1}{2X_2} \leq t\bigg)-P(X_3<t)\bigg)dF(t)\\
   &= \int_{-\infty}^\infty \big(G(t)- F(t)\big)dF(t).
\end{align*}
The measure $\Delta$ represents the integral of the difference between two cumulative distribution functions over the entire real line. 
For \( X_1 \), \( X_2 \), and \( X_3 \), 
\begin{align*}
  \int_{-\infty}^\infty G(t) dF(t)&= \int_{-\infty}^\infty P\left(\dfrac{X_1 X_2 - 1}{2 X_2} \leq t \right) dF(t)\\
  &=  P\left(\dfrac{X_1 X_2 - 1}{2 X_2} \leq X_3 \right) \\
  &= E\left(\mathrm{I}\left\{ \dfrac{X_1 X_2 - 1}{2 X_2}  \leq X_3 \right\}\right), 
\end{align*}
and if \( F \) is a continuous d.f. then
$\int_{-\infty}^\infty F(t) dF(t) = \frac{1}{2}.$
In view of the characterization, the following empirical $U$-Statistics-based test statistic is considered:
\begin{equation}\label{teststats} 
    \Delta_n = \int_{-\infty}^\infty (G_n(t)-F_n(t)) dF_n(t),
\end{equation}
 where
\begin{equation*}
    G_n(t) = \binom{n}{2}^{-1}\sum_{i=1}^n \sum_{j=1;j<i}^n \mathrm{I}\bigg\{  \dfrac{X_iX_j-1}{2X_j} \leq t \bigg\}~~ \text{and} ~~ F_n(t)= \dfrac{1}{n}\sum_{i=1}^n \mathrm{I}\{X_i\leq t \}
\end{equation*}
are $U$-empirical d.f's. 
Here, $G_n(t)$ represents an empirical distribution function based on pairs of observations, while \( F_n(t) \) is the empirical distribution function based on individual observations. These $U$-empirical d.f.'s are used to approximate the underlying distribution functions that appear in the definition of \( \Delta \). The difference between \( G_n(t) \) and \( F_n(t) \) quantifies how the data diverge from the hypothesized model, and the test statistic \( \Delta_n \) is constructed as their weighted integral over the entire range of \( t \). 

After integration, an asymptotic equivalence test statistic of $\Delta_n$ becomes
\begin{align}\label{Dn}
\widehat{\Delta}^*_n  &= \binom{n}{3}^{-1}\sum_{i=1}^n \sum_{j=1;j<i}^n\sum_{k=1;k<j}^n \mathrm{I}\bigg\{ \dfrac{X_iX_j-1}{2X_j} \leq X_k \bigg\} -\dfrac{1}{2}.
\end{align}
 The estimator \( \widehat{\Delta}^*_n \) is unbiased for \( \Delta \).
 The symmetric kernel of \( \widehat{\Delta}^*_n \) is given by 
\begin{align*}
    &h(X_1, X_2, X_3)\\
    &=\dfrac{1}{3}\bigg(\mathrm{I}\bigg( \dfrac{X_1X_2-1}{2X_2}\leq X_3\bigg)+\mathrm{I}\bigg( \dfrac{X_1X_3-1}{2X_3}\leq X_2\bigg) +\mathrm{I}\bigg( \dfrac{X_2X_3-1}{2X_3}\leq X_1\bigg) \bigg)-\dfrac{1}{2}.
\end{align*}
 Next, the asymptotic properties of the test statistics are studied.

\begin{theorem}\label{thm2}
    Let $X, X_1, X_2, X_3$ be independent and identically distributed random
variables with d.f. $F$. As $n\to \infty$, $\sqrt{n}(\widehat{\Delta}^*_n -\Delta)$ converges in distribution to a normal with mean zero and variance $9\sigma^2$, where $\sigma^2$ is given by 
\begin{equation}\label{sigma}
    \sigma^2= Var\big( K(X) \big)
\end{equation}
\begin{align*}
\ \ \mbox{with}\ \    K(x)& = \dfrac{2}{3} P\Bigg( \dfrac{xX_2-1}{2X_2}\leq X_3\Bigg) + \dfrac{1}{3} P\Bigg( \dfrac{X_2X_3-1}{2X_3}\leq x\Bigg).  
\end{align*}
\end{theorem}

\begin{proof} Let $\widehat{\Delta}^*_n= \widehat{\Delta}^{'}_{n}-\dfrac{1}{2}$, where 
    \begin{equation}
    \widehat{\Delta}^{'}_{n}= \binom{n}{3}^{-1}\sum_{i=1}^n \sum_{j=1;j<i}^n\sum_{k=1;k<j}^n \mathrm{I}\Big\{\dfrac{X_iX_j-1}{2X_j}\leq X_k \Big\}.
\end{equation}
The estimator $\widehat{\Delta}^{'}_n$ is an unbiased estimator of $\Delta^{'}$. 
 Hence the asymptotic distribution of $\sqrt{n}(\widehat{\Delta}^*_{n} -\Delta)$ and $\sqrt{n}(\widehat{\Delta}^{'}_{n} - {\Delta}^{'})$ are the same. 
 Note that $\widehat{\Delta}^{'}_{n}$ is a $U$-Statistic with kernel function $h^*(X_1, X_2, X_3)= \mathrm{I}\Big( \dfrac{X_1X_2-1}{2X_2}\leq X_3\Big)$ of degree three. The symmetric version of the kernel $h^*(\cdot)$ is given by 
\begin{align}\label{h1}
     h_1(X_1, X_2, X_3)  &=\dfrac{1}{3}\bigg(\mathrm{I}\bigg( \dfrac{X_1X_2-1}{2X_2}\leq X_3\bigg)+\mathrm{I}\bigg( \dfrac{X_1X_3-1}{2X_3}\leq X_2\bigg) +\mathrm{I}\bigg( \dfrac{X_2X_3-1}{2X_3}\leq X_1\bigg) \bigg).
\end{align}
Using the central limit theorem for $U$-Statistics (see \cite{lee2019u}, Theorem 1, Page 76), as $n \to \infty$, $\sqrt{n}(\widehat{\Delta}^{'}_{n} -\Delta^{'})$ converges to normal with mean zero and variance $9\sigma^2$, where $\sigma^2$ is the asymptotic variance of $\widehat{\Delta}^{'}_{n}$ and is given by 
\begin{equation*}
    \sigma^2= Var\big(E(h_1(X_1, X_2, X_3)|X_1)\big).
\end{equation*}
Using ($\ref{h1}$), we find the conditional expectation as
\begin{align*}
   K(x)&= E(h_1(X_1, X_2, X_3)|X_1=x) = \dfrac{2}{3} P\Bigg( \dfrac{xX_2-1}{2X_2}\leq X_3\Bigg) + \dfrac{1}{3} P\Bigg( \dfrac{X_2X_3-1}{2X_3}\leq x\Bigg).  
\end{align*}
Hence, we obtain the variance expression as specified in the theorem. 
\end{proof}

 The following result shows that the asymptotic distribution of the null distribution follows a normal distribution. Note that under $H_0$, $\Delta =0$. 
\begin{corollary}\label{cor1}
    Let $X$ be a random variable with the d.f. specified in (\ref{cdf}). As $n\to \infty$, $\sqrt{n}\widehat{\Delta}^*_{n}$ converges in distribution to a normal with mean zero and variance $9\sigma^2_0$, where $\sigma^2_0$ is given by
    \begin{equation}\label{sigma_0}
        \sigma^2_0=    Var\big(K^*(X)\big),
    \end{equation}
    \begin{align*}
  \ \  \mbox{where} \ \  K^*(x)&= \dfrac{2}{3} P\Bigg( \dfrac{xX_2-1}{2X_2}\leq X_3\Bigg) + \dfrac{1}{3} P\big( X_1\leq x\big). 
    \end{align*}
\begin{proof}
    Under $H_0$, using the characterization of the Cauchy distribution given in Theorem \ref{thm1}, we obtain 
\begin{equation*}
    P\bigg( \dfrac{X_1X_2-1}{2X_2}\leq x \bigg)=P(X_1\leq x).
\end{equation*}
Now, we know that $\sigma^2= Var(E(h(X_1, X_2, X_3)|X_1))$.
Hence, from (\ref{sigma}), we have 
\begin{align*}
   K^*(x)=& E(h_1(X_1, X_2, X_3)|X_1=x) = \dfrac{2}{3} P\Bigg( \dfrac{xX_2-1}{2X_2}\leq X_3\Bigg) + \dfrac{1}{3} P\big( X_1\leq x\big).  
\end{align*}
    from where follows (\ref{sigma_0}).
\end{proof}
\end{corollary}

Thus, given Corollary \ref{cor1}, we obtain a test based on normal approximation, and we reject the null hypothesis $H_0$ against the alternative hypothesis $H_1$ if 
\begin{equation*}
    \dfrac{\sqrt{n}|\widehat{\Delta}^*_{n}|}{\widehat{\sigma}_0}\geq Z_{\alpha/2},
\end{equation*}
where $\widehat{\sigma}_0$ is a consistent estimator of $\sigma_0$ and $Z_\alpha$ is the upper $\alpha$-percentile point of the standard normal distribution. As a result, implementing the test using normal approximation is challenging due to the difficulty in finding a consistent estimator of \(\sigma^2_0\). This challenge motivates the development of a nonparametric empirical likelihood-based test. Next, we discuss the JEL and AJEL-based tests for testing the Cauchy distribution.
 
\subsection{JEL ratio test}\label{JEL}
\noindent In this subsection, we construct a JEL test based on jackknife pseudo values that is constructed using (\ref{Dn}). The jackknife pseudo-values are given by 
\begin{equation}\label{pseudo}
    \widehat{J}_i= n\widehat{\Delta}^*_{n} - (n-1)\widehat{\Delta}_{n-1}^{*(-i)},~~~ i=1,2,\ldots,n,
\end{equation}
where the value $\widehat{\Delta}_{n-1}^{*(-i)}$ is computed using $ \widehat{\Delta}^*_{n}$ based on the $(n-1)$ observations $X_1, \ldots, X_{i-1}$, $X_{i+1}, \ldots, X_n$. Hence, it can be easily shown that $\widehat{\Delta}^*_{n}= \dfrac{1}{n}\sum\limits_{i=1}^n\widehat{J}_i$. Although  $\widehat{J}_i$ are typically dependent random variables, they are asymptotically independent under mild conditions (see \cite{shi1984approximate}). This allows us to use the empirical likelihood method using the Jackknife pseudo-value. 
Let $p_i$ be the probability assigned to each $\widehat{J}_i$. 
Hence, the JEL based on $\Delta$ is given by
\begin{equation*}
    \mathcal{L}(\Delta)= \max \bigg\{ \prod_{i=1}^n p_i: \sum_{i=1}^n p_i=1, \sum_{i=1}^n p_i\widehat{J}_i=0 \bigg\}.
\end{equation*}
We know that $\prod_{i=1}^np_i$ subject to $\sum_{i=1}^np_i=1$ attains its maximum value $\dfrac{1}{n^n}$ when $p_i= \dfrac{1}{n}$. Hence, the JEL ratio for $\Delta$ is given by
\begin{equation*}
    \mathcal{R}(\Delta)= \dfrac{\mathcal{L}(\Delta)}{n^{-n}}= \max \bigg\{ \prod_{i=1}^n np_i: \sum_{i=1}^n p_i=1, \sum_{i=1}^n p_i\widehat{J}_i=0 \bigg\}.
\end{equation*}
By using the standard Lagrange multiplier method, the above maximization is attained at 
\begin{equation*}
    \widehat{p}_i= \dfrac{1}{n}\dfrac{1}{(1+\widehat{\lambda}_1 \widehat{J}_i)}
\end{equation*}
where $\widehat{\lambda}_1$ satisfies 
\begin{equation*}
    \dfrac{1}{n}\sum_{i=1}^n \dfrac{\widehat{J}_i}{(1+\lambda_1 \widehat{J}_i)}=0.
\end{equation*}
Hence, we obtain the jackknife empirical log-likelihood ratio as
\begin{equation*}
    \log\mathcal{R}(\Delta)=- \sum_{i=1}^n \log \big(1+\widehat{\lambda}_1\widehat{J}_i \big).
\end{equation*}
Next, the asymptotic null distribution of the JEL ratio test is determined to construct a critical region for the JEL-based test. Following Theorem $2$ of \cite{jing2009jackknife}, the well-known Wilks' theorem regarding the asymptotic null distribution of the likelihood ratio statistic is applied.
\begin{theorem}\label{jelthm}
    Assume that
    \begin{enumerate}
        \item $E[h^2(X_1,X_2,X_3)]<\infty$, and
        \item $\sigma_1^2= Var[E\{h(X_1, X_2, X_3)|X_1\}]>0$.
    \end{enumerate}
    Then, as \( n \to \infty \), \( -2 \log \mathcal{R}(\Delta) \) converges in distribution to a chi-squared distribution with one degree of freedom.
\end{theorem}
In view of Theorem \ref{jelthm}, the null hypothesis \( H_0 \) is rejected in favor of the alternative hypothesis \( H_1 \) at the \(\alpha\) level of significance, if 
\begin{equation*}
    -2\log\mathcal{R}(\Delta)>\chi^2_{1,\alpha},
\end{equation*}
where $\chi^2_{1,\alpha}$ is the upper $\alpha$-percentile point of a $\chi^2$ distribution with one degree of freedom. 

\subsection{AJEL ratio test}\label{AJEL}
\noindent The convex hull of $\widehat{J}_i-E(\widehat{J}_i)$ may not contain zero when the sample size is very small. To overcome this difficulty, \cite{chen2016adjusted} combined the idea of jackknife and adjusted jackknife empirical likelihood and proposed an adjusted jackknife empirical likelihood (AJEL) ratio test. In the AJEL method, \cite{chen2016adjusted} proposed incorporating an additional jackknife pseudo-value alongside the existing jackknife pseudo-values given in (\ref{pseudo}) and defined as
\begin{equation*}
    \widehat{J}_{n+1}=\dfrac{-k_n}{n}\sum_{i=1}^n\widehat{J}_i. 
\end{equation*}
In the context of AJEL, \cite{chen2008adjusted} suggest the most popular choice of $k_n$ as $\max\{1, \log(n)/2\}$. Hence, we consider the same value for $k_n$.
The AJEL is defined as follows,
\begin{equation*}
    \mathcal{L}(\Delta)= \max\bigg\{\prod_{i=1}^{n+1} p_i: \sum_{i=1}^{n+1} p_i=1, p_i\geq 0, \sum_{i=1}^{n+1}p_i \widehat{W}_i=0 \bigg\},
\end{equation*}
where 
\begin{equation}\label{wi}
    \widehat{W}_i= \begin{cases}
\widehat{J}_i  \hspace{5.7cm}  \text{for}  ~ i=1,2,\ldots,n. \\
-\max\Big(1,\dfrac{1}{2}\log(n)\Big) \dfrac{1}{n} \sum_{i=1}^n \widehat{J}_i \hspace{0.7cm} ~~~ \text{for} ~ i=n+1.
\end{cases}
\end{equation}
Hence, the AJEL ratio for $\Delta$ is given by
\begin{equation*}
    \mathcal{R}(\Delta)= \max\bigg\{\prod_{i=1}^{n+1} (n+1)p_i: \sum_{i=1}^{n+1} p_i=1, p_i\geq 0, \sum_{i=1}^{n+1}p_i \widehat{W}_i=0 \bigg\}.
\end{equation*}
Thus, the next theorem establishes that the asymptotic distribution of the empirical likelihood ratio test statistic follows a chi-squared distribution, analogous to Wilks' theorem.
\begin{theorem}\label{ajelth}
Under the assumptions outlined in Theorem \ref{jelthm}, under $H_0$, $-2\log\mathcal{R}(\Delta)$ converges in distribution to $\chi^2$ with one degree of freedom, where the log empirical likelihood ratio is obtained as
    \begin{equation*}
        \log \mathcal{R}(\Delta)= -\sum_{i=1}^{n+1} \log \Big(1+\widehat{\lambda}_2 \widehat{W}_i \Big)
    \end{equation*}
where $\widehat{\lambda}_2$ satisfies
\begin{equation*}
    \dfrac{1}{n+1}\sum_{i=1}^{n+1} \dfrac{\widehat{W}_i}{1+\lambda_2 \widehat{W}_i}=0, 
\ \ \mbox{and} \ \ \
    \widehat{p}_i= \dfrac{1}{(n+1)}\dfrac{1}{(1+\widehat{\lambda}_2 \widehat{W}_i)}.
\end{equation*}
\end{theorem}
\begin{proof}
    Using the jackknife pseudo value defined in (\ref{wi}), we obtain $Z_n=\dfrac{1}{n}\sum_{i=1}^n \widehat{J}_i$. Define $S^2=\dfrac{1}{n}\sum_{i=1}^n (\widehat{J}_i)^2$, hence by the strong law of large numbers, we have $S^2= \sigma^2+ O_p(1)$. We have $|\widehat{\lambda}_2|=\mathcal{O}_p(1/\sqrt{n})$, when $k_n=O_p(n)$. 
Consider 
\begin{align*}
    -2 \log\mathcal{R}(\Delta)&= \sum_{i=1}^{n+1} 2\log (1 +\widehat{\lambda}_2 \widehat{W}_i)\\
    &= 2\sum_{i=1}^{n+1} \big( \widehat{\lambda}_2 \widehat{W}_i - (\widehat{\lambda}_2 \widehat{W}_i)^2/2 \big) +O_p(1)\\
    &= 2n\widehat{\lambda}_2 Z_n -nS\widehat{\lambda}_2^2+ O_p(1)\\
    &= \dfrac{n Z_n^2}{S^2}+ O_p(1),
\end{align*}
where the preceding identity is established because the $(n + 1)$th term in the summation is $k_nO_p(n^{-3/2}) = O_p(n) \mathcal{O}_p(n^{-3/2}) =O_p(1)$. Hence by Slutsky's theorem, as $n\to \infty$,
\begin{equation*}
    -2 \log\mathcal{R}(\Delta)=\dfrac{n Z_n^2}{S^2} \xrightarrow{d} \chi^2_1. 
\end{equation*} 
\end{proof}
Thus in view of Theorem \ref{ajelth}, we reject the null hypothesis $H_0$ against the alternative hypothesis $H_1$ at a $\alpha$ level of significance if
\begin{equation*}
     -2 \log\mathcal{R}(\Delta)> \chi^2_{1, \alpha}. 
\end{equation*}

\section{Simulation study}\label{SStudy}
\noindent In this section, a comprehensive simulation study is presented to evaluate the empirical power of the new tests relative to established methods. All computations for this and the subsequent sections were carried out using the statistical software R. The simulations were conducted at a significance level \( \alpha = 0.05 \) and for sample sizes \( n = 20, 40, 60, 80, \) and \( 100 \).
To evaluate the empirical power of the newly proposed tests in comparison with the following established tests for the Cauchy distribution:
\begin{itemize}
    \item The classical Kolmogorov-Smirnov $(KS_n)$ test:
    \begin{equation*}
        KS_n= \max\Big[\max \limits_{1\leq i \leq n} \Big\{\dfrac{i}{n}-F(x), F(x)- \dfrac{i-1}{n} \Big\}\Big],
    \end{equation*}
    where $F(x)$ is the  c.d.f. defined in $(\ref{cdf})$.  
    \item Cram\'{e}r-von Mises $(CM_n)$ test: 
     \begin{equation*}
         CM_n= \sum_{i=1}^n \Big( F(X_{(i)}-\dfrac{i-0.5}{n} \Big)^2 +\dfrac{1}{12n}.
     \end{equation*}
    \item Modified Anderson-Darling $(MA_n)$ test:
    \begin{equation*}
         MA_n= \dfrac{-2}{n} \sum_{i=1}^n \Big[(i-0.5)\log \{F(X_{(i)})\} + (n-i+0.5)\log \{1-F(X_{(i)})\}   \Big]-n,
     \end{equation*}
   where $X_{(\cdot)}$ denotes the order statistics.
    \item Likelihood ratio based, \cite{zhang2002powerful} tests:
    \begin{align*}
        ZA_n&= -\sum_{i=1}^n \bigg[\dfrac{\log\{F(X_{(i)}) \}}{n-i+0.5}+ \dfrac{\log\{1-F(X_{(i)}) \}}{i-0.5}  \bigg],\\
        ZB_n&=\sum_{i=1}^n \bigg[\log \Big\{ \dfrac{1/F(X_{(i)})-1}{(n-0.5)/(i-0.75)-1} \Big\} \bigg]^2,\\
        ZC_n&= \max \limits_{1\leq i \leq n} \Big[ (i-0.5) \log \Big\{\dfrac{i-0.5}{nF(X_{(i)})} \Big\} \\
        &\hspace{1.5cm}+ (n-i+0.5) \log \Big\{\dfrac{n-i+0.5}{n(1-F(X_{(i)}))} \Big\} \Big].
    \end{align*}
    
    \item \cite{gurtler2000goodness} empirical characteristic function based test:
  \begin{equation*}
      D_{n,\lambda}= \dfrac{2}{n}\sum_{i,j=1}^n \dfrac{\lambda}{\lambda^2+(X_i-X_j)^2}- 4\sum_{i=1}^n \dfrac{1+\lambda}{(1+\lambda)^2+X_i^2}+\dfrac{2n}{2+\lambda},
  \end{equation*}
    where $\lambda>0$ be the weighting parameter. Based on the \cite{gurtler2000goodness} simulation results, using  $\lambda= 5$  proves to be effective in practice, leading to a robust test.
    
    \item \cite{kullback1997information} distance $(KL_n)$ test:
   The Kullback-Leibler (KL) distance is the most widely used information criterion for evaluating model discrepancies. \cite{mahdizadeh2017new} suggested a test statistic 
    \begin{equation*}
        KL_n= \exp \bigg\{-HV_{n,m}-\dfrac{1}{n} \sum_{i=1}^n\log(f(X_{i})) \bigg\},
    \end{equation*}
where $m= n/2$ is the optimal window size mentioned in \cite{mahdizadeh2019goodness}, $f(x)$ be the p.d.f defined in $(\ref{pdf})$ and 
\begin{equation*}
    HV_{n,m}= \dfrac{1}{n}\sum_{i=1}^n\log \Big\{\dfrac{n}{2m}\big( X_{(i+m)}-X_{(i-m)}\big) \Big\}. 
\end{equation*}
    
    \item Entropy estimator based test:  Recently  \cite{mahdizadeh2019goodness} proposed six tests based on entropy measure. In our simulation study, we consider the test which gives higher power than other
    tests. Hence we consider the test given by
\begin{equation*}
    HE_{n,m} =\dfrac{1}{n} \sum_{i=1}^n \log \Big\{\dfrac{n}{c_im}\big( Y_{(i+m)}-Y_{(i-m)}\big)\Big\}, 
\end{equation*}
    where 
    \begin{equation*}
        c_i=  \begin{cases}
        1+\dfrac{1+i}{m}-\dfrac{i}{m^2} \hspace{1.15cm} ~~ ~ 1\leq i \leq m, \\
        2  \hspace{3.1cm} ~ m+1\leq i \leq n-m,\\ 
        1+\dfrac{n-i}{m+1} \hspace{1.5cm} ~ n-m+1\leq i \leq n
    \end{cases}
    \end{equation*}
  and the ${Y_{(i)}}'s $ are
   \begin{equation*}
       \begin{cases}
       Y_{(i-m)}= p +\dfrac{i-1}{m}\big(X_{(1)}-p \big) \hspace{1.15cm} ~~ ~ 1\leq i \leq m, \\
       Y_{(i)}= X_{(1)}     \hspace{3.9cm} ~ m+1\leq i \leq n-m,\\
       Y_{(i+m)}= q+\dfrac{n-i}{m}\big(q-X_{(n)}\big)    \hspace{0.7cm} ~ n-m+1\leq i \leq n
   \end{cases}
   \end{equation*}
\noindent where the constants $p$ and $q$ are to be determined such that $P(p\leq X\leq q)\approx 1$ and $m= n/2$ is the optimal window size.
\end{itemize}

First, we find an empirical type-I error rate. For this purpose, we generate observations from a standard Cauchy distribution. The empirical type-I error rates for the tests based on JEL and AJEL are computed using 10,000 replications. The proportion of times the null hypothesis is rejected at the \( 5\% \) significance level is assessed. The results, presented in Table \ref{table:sizetable}, show that the empirical type-I error rates converge to the specified significance level $\alpha$ as the sample size increases.
Next, the empirical power of the proposed JEL and AJEL tests is assessed by examining their performance against the alternative distributions considered below.
\begin{itemize}
    \item Student's t distribution with $r$ degrees of freedom: $t_r$ 
    \item Normal distribution with parameters $\mu$ and $\sigma^2$: $N(\mu, \sigma^2)$
    \item Gamma distribution with parameters $\alpha$ and $\beta$: $G(\alpha,\beta)$
    \item Laplace distribution with parameters $\mu$ and $\sigma^2$: $L(\mu, \sigma^2)$
    \item Beta distribution with parameters $\alpha$ and $\beta$: $B(\alpha,\beta)$
    \item Uniform distribution with parameters $\alpha$ and $\beta$: $U(\alpha,\beta)$
\end{itemize}
The alternative distributions considered are \( t_3 \), \( N(0,1) \), \( G(2,1) \), \( L(0,1) \), \( B(2,2) \), and \( U(0,1) \). Samples of sizes \( n = 20, 40, 60, 80 \), and \( 100 \) are generated from each of these distributions. The empirical power of the tests based on JEL and AJEL is calculated using 10000 replications with the proportion of times the null hypothesis is rejected at the $5\%$ significance level. The results are presented in Tables \ref{table:powertable1}. From Tables \ref{table:powertable1}, we observed that the proposed test has very good power, even for a relatively small sample size.

\begin{table}[!ht]
\centering
\caption{type-I error rates at $0.05$ level of significance}
\label{table:sizetable} 
\resizebox{14.5cm}{!}{
\fontsize{10pt}{13pt}\selectfont
\begin{tabular}{p{0.7cm} p{0.5cm} p{0.7cm} p{0.7cm} p{0.7cm} p{0.7cm} p{0.7cm}p{0.7cm} p{0.7cm} p{0.7cm} p{0.7cm} p{0.7cm} p{0.7cm} p{0.7cm} } 
\hline
        $n$& JEL & AJEL  & $KS_n$  & $CM_n$  &$MA_n$  & $ZA_n$   & $ZB_n$   &  $ZC_n$  & $D_{n,\lambda}$  & $KL_n$& $HE_{n,m}$  \\[0.3ex]
\hline
          20&0.044&0.030& 0.057& 0.048& 0.047& 0.048& 0.045& 0.039& 0.053& 0.035&0.035 \\[0.5ex]
          40&0.032&0.029& 0.044& 0.051& 0.046& 0.061& 0.060& 0.059 &0.045& 0.054&0.040\\[0.5ex]
          60&0.047&0.035& 0.069& 0.078&  0.073& 0.047& 0.047& 0.043& 0.058& 0.044&0.042\\[0.5ex]
          80&0.049&0.047& 0.067& 0.061&  0.062& 0.064& 0.039& 0.041& 0.055& 0.064&0.044\\[0.5ex]
         100&0.051&0.049& 0.078& 0.069 & 0.091& 0.058& 0.059& 0.068& 0.058& 0.056&0.038\\[0.5ex]
\hline
\end{tabular}
}
\end{table}

\begin{table}[!ht]
\centering
\caption{Empirical power compassion at $0.05$ level of significance}
\label{table:powertable1} 
\resizebox{14.5cm}{!}{
\fontsize{10pt}{13pt}\selectfont
\begin{tabular}{p{0.7cm} p{0.5cm} p{0.7cm} p{0.7cm} p{0.7cm} p{0.7cm} p{0.7cm}p{0.7cm} p{0.7cm} p{0.7cm} p{0.7cm} p{0.7cm} p{0.7cm} p{0.7cm} } 
\hline
      Alt.  &$n$& JEL & AJEL  & $KS_n$  & $CM_n$ &$MA_n$  & $ZA_n$   & $ZB_n$   &  $ZC_n$  & $D_{n,\lambda}$  & $KL_n$& $HE_{n,m} $  \\[0.3ex]
\hline
           $t_3$ & 20&0.553&0.482& 0.028& 0.035& 0.019& 0.006& 0.001& 0.018& 0.006 &0.099&0.043 \\[0.5ex]
             & 40&0.750&0.717& 0.022& 0.046&  0.002& 0.046& 0.019& 0.169& 0.033& 0.527&0.473\\[0.5ex]
             & 60&0.865&0.852& 0.013& 0.044&  0.007& 0.424& 0.198& 0.398& 0.246& 0.741&0.813\\[0.5ex]
             & 80&0.953&0.944& 0.014& 0.021&  0.006& 0.594& 0.445& 0.600& 0.384& 0.887&0.893\\[0.5ex]
            & 100&0.968&0.965& 0.016& 0.018&  0.019& 0.862& 0.705& 0.589& 0.562& 0.909&0.911\\[1.5ex]
      N(0,1) & 20&0.512&0.447& 0.002& 0.001&  0.003& 0.008& 0.008& 0.030& 0.005& 0.193&0.134\\[0.5ex]
       & 40&0.686&0.644& 0.001& 0.002&  0.003& 0.549& 0.299& 0.509 &0.398 &0.798&0.686\\[0.5ex]
       &  60&0.836&0.817& 0.001& 0.004&  0.027& 0.799 &0.853& 0.776& 0.841& 0.887&0.764\\[0.5ex]
        & 80&0.923&0.916& 0.005& 0.010&  0.039& 0.830& 0.910& 0.807& 0.908& 0.908&0.884\\[0.5ex]
       & 100&0.964&0.960& 0.002& 0.005 &0.076& 0.919& 0.948& 0.923& 0.938& 0.934&0.901\\[1.5ex]
               G(2,1) & 20&0.718&0.484& 0.010 &0.016& 0.015&  0.077& 0.021 &0.009& 0.023& 0.168&0.091\\[0.5ex]
        & 40 &1.000&1.000&0.054& 0.247& 0.124&  1.000 &0.535& 0.042& 0.055& 0.917&0.910\\[0.5ex]
        &  60 &1.000&1.000&0.193& 0.600& 0.383&  1.000& 0.960& 0.277& 0.322& 0.990&0.995\\[0.5ex]
         & 80&1.000&1.000& 0.341 &0.809& 0.592&  1.000& 0.999& 0.663& 0.456& 1.000&1.000\\[0.5ex]
         &100&1.000&1.000& 0.476& 0.955& 0.794&  1.000& 1.000& 0.878& 0.679& 1.000&1.000\\[1.5ex]
     L(0,1)&  20&0.540&0.477& 0.022& 0.011&  0.009& 0.005& 0.006& 0.009& 0.006& 0.066&0.023\\[0.5ex]
      &  40&0.742&0.720& 0.006& 0.002&  0.007& 0.051& 0.052& 0.052& 0.073& 0.681&0.402\\[0.5ex]
      &  60&0.885&0.871& 0.002& 0.002&  0.004& 0.336& 0.196& 0.346& 0.244& 0.729&0.836\\[0.5ex]
       & 80&0.932&0.924& 0.002& 0.004&  0.009& 0.800& 0.390& 0.781& 0.400& 0.887&0.902\\[0.5ex]
      & 100&0.974&0.970& 0.001& 0.002&  0.013& 0.942& 0.684& 0.879& 0.660& 0.948&0.967\\[1.5ex]
         B(2,2)&  20&1.000& 1.000& 0.020& 0.033&  0.050& 0.229& 0.170& 0.013& 0.013& 0.514&0.224\\[0.5ex]
         &40&1.000& 1.000& 0.149& 0.110& 0.216& 0.749& 0.971 &0.809& 0.194& 0.879&0.768\\[0.5ex]
         &  60&1.000& 1.000& 0.706& 0.668& 0.794& 0.939& 1.000& 0.990& 0.964& 0.959&0.961\\[0.5ex]
          & 80&1.000& 1.000& 1.000& 0.959&  0.991& 1.000& 1.000& 1.000& 1.000& 1.000&1.000\\[0.5ex]
         & 100&1.000& 1.000& 1.000& 0.997 & 1.000& 1.000& 1.000& 1.000& 1.000& 1.000&1.000\\[1.5ex]
      U(0,1)&  20&1.000&1.000& 0.006&  0.012& 0.014& 0.228& 0.132& 0.018& 0.006& 0.918&0.635\\[0.5ex]
      &  40&1.000&1.000& 0.809& 0.226&  0.246& 0.820& 0.910& 0.801& 0.615& 1.000&0.920\\[0.5ex]
      &  60&1.000&1.000& 1.000& 0.739&  0.882& 1.000& 1.000& 1.000& 0.998& 1.000&1.000\\[0.5ex]
      &  80&1.000&1.000& 1.000& 0.912&  0.962& 1.000& 1.000 &1.000& 1.000& 1.000&1.000\\[0.5ex]
      & 100&1.000&1.000& 1.000& 0.998 & 0.998& 1.000& 1.000& 1.000& 1.000& 1.000&1.000\\
\hline
\end{tabular}}
\end{table}

Compared to classical tests, the JEL and AJEL tests show better performance. Classical tests often struggle with small sample sizes or specific types of data, which can make them less reliable in some cases. However, the JEL and AJEL tests are more dependable and effective across a wide range of situations. They handle different sample sizes and various data distributions well, making them a versatile and robust choice for many different scenarios. In real-world situations, where conditions can vary greatly, the JEL and AJEL tests provide a more stable and dependable method for detecting significant differences.

\section{Data analysis}\label{sec4}
\noindent Next, we analyze two real-world data sets employing all the proposed methodologies.
\subsection*{I.German Stock Index Data}
\noindent Stock market return distributions frequently exhibit heavy tails, a feature not commonly associated with the normal distribution. Consequently, it is more appropriate to analyze stock market return data using Cauchy distributions (\cite{khan2024testing}), which are known for their ability to model data with heavy tails. In this study, we examine the returns of the closing prices of the German Stock Index using the Cauchy distribution. The DAX (Deutscher Aktienindex) is a stock market index that includes the $30$ largest blue-chip companies listed on the Frankfurt Stock Exchange, representing a significant segment of the German equity market.

In this study, we apply goodness-of-fit tests to a dataset consisting of $30$ returns of closing prices from the DAX (Deutscher Aktienindex). These data points were collected daily, starting from January $1, 1991$, while excluding weekends and public holidays. The dataset, rounded to seven decimal places, is provided in Table \ref{table:data1} and sourced from the \texttt{datasets} package in the R statistical software. Figure \ref{fig:data1}  displays the histogram with an overlaid Cauchy density function.
\begin{table}[!ht]
    \centering
    \caption{ Closing prices of DAX}
    \label{table:data1} 
    \begin{tabular}{c c c c c c }
    \hline
    0.0011848 & -0.0057591 & -0.0051393 & -0.0051781 & 0.0020043 & 0.0017787 \\
    0.0026787 & -0.0066238 & -0.0047866 & -0.0052497 & 0.0004985 & 0.0068006 \\
    0.0016206 & 0.0007411 & -0.0005060 & 0.0020992 & -0.0056005 & 0.0110844 \\
    -0.0009192 & 0.0019014 & -0.0042364 & 0.0146814 & -0.0002242 & 0.0024545 \\
    -0.0003083 & -0.0917876 & 0.0149552 & 0.0520705 & 0.0117482 & 0.0087458 \\
    \hline
    \end{tabular} 
\end{table}

The values of JEL test statistic yield $0.3206$ with a p-value $0.5711$ and for AJEL test statistics is $0.3078$ with a p-value $0.5613$, whereas all other statistics are computed and presented in Table $9$ of \cite{mahdizadeh2019goodness}. For each test performed, the null hypothesis that the data are drawn from a Cauchy distribution cannot be rejected at the $0.05$ significance level.

\begin{figure}[!ht]
    \centering
\includegraphics[width=8.5cm,height=8cm]{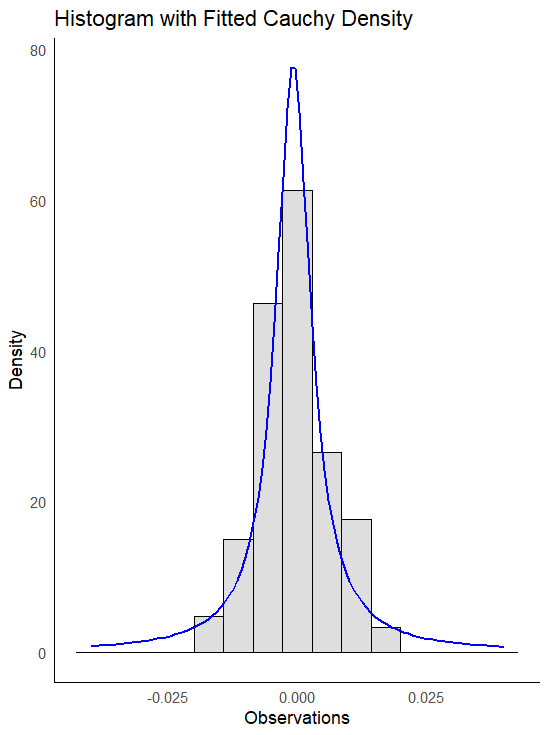}
    \caption{Q-Q plot and histogram for German Stock Index (DAX) Data}
    \label{fig:data1}
\end{figure}


\subsection*{II. Cryptocurrency Data}
\noindent Now, we utilize tests for the Cauchy distribution to analyze the log returns of EOS cryptocurrency.  Cauchy distribution is
found to be a comparably good model for such data sets,
see \cite{ebner2024cauchy}. The dataset analyzed in this study comprises daily closing prices of EOS in USD, spanning from December $25, 2020$, to July $07, 2024$. The data were accessible from CryptoDataDownload (\url{www.cryptodatadownload.com/data}). For the analysis, we utilize the daily returns $r_d=(P_d-P_{d-1})/P_{d-1}$ of the closing prices, where $P_d$ be the closing price of the $d$-th day. Figure \ref{fig:data2}  displays the histogram of the dataset along with the fitted Cauchy distribution densities. Visually, the Cauchy model appears to be a reasonable approximation. 
The calculated value of the JEL statistic is $0.2084$ with a corresponding p-value of $0.6479$, while the AJEL test statistic value is $0.1192$ with a p-value of $0.6464$. Based on JEL and AJEL values, we fail to reject the null hypothesis at the $0.05$ significance level.  

\begin{figure}[!ht]
    \centering
    \includegraphics[width=8.5cm,height=8cm]{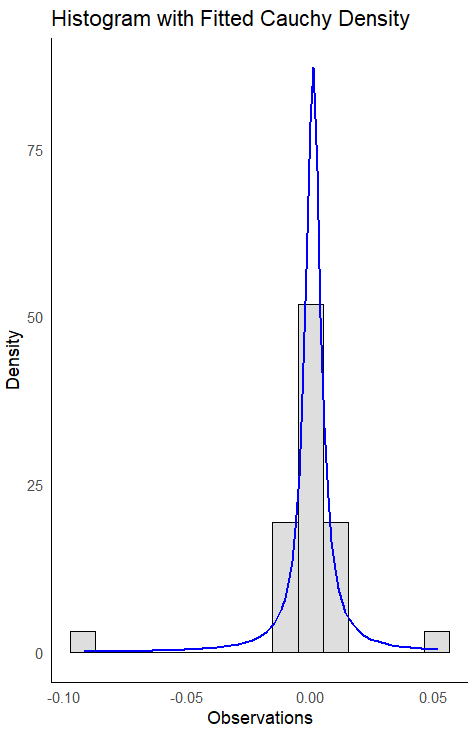}
    \caption{Q-Q plot and histogram for daily closing prices of Bitcoin data}
    \label{fig:data2}
\end{figure}

\section{Summary}\label{conclusion} 
\noindent In this paper, jackknife and adjusted jackknife empirical likelihood ratio tests are proposed, using $U$-Statistic theory and based on a straightforward characterization of the Cauchy distribution. Drawing inspiration from the work of \cite{jing2009jackknife}, the paper derives the asymptotic distribution of the empirical likelihood ratio test statistic based on the JEL method, analogously to Wilks' theorem. This result is then extended to the AJEL method. A simulation study is conducted to evaluate the efficacy of the proposed test procedures. Additionally, the methodologies are applied to real data sets, specifically the daily closing prices of the German stock index and the EOS cryptocurrency. The data analysis shows that both data sets exhibit heavy-tailed behaviour, suggesting that the Cauchy distribution is suitable for both data sets. 

\end{document}